\documentclass[11pt,oneside,leqno]{amsart}

\usepackage{amsxtra}
\usepackage{amsopn}
\usepackage{amsmath,amsthm,amssymb}
\usepackage{amscd}
\usepackage{color}
\usepackage{amsfonts}
\usepackage{latexsym}
\usepackage{verbatim}

\theoremstyle{plain}
\newtheorem{theorem}{Theorem}[section]
\newtheorem{lemma}[theorem]{Lemma}
\newtheorem{prop}[theorem]{Proposition}

\theoremstyle{definition}
\newtheorem{definition}[theorem]{Definition}
\newtheorem{rem}[theorem]{Remark}
\newtheorem{ex}[theorem]{Example}

\newcommand\C{{\mathbb C}}

\newcommand\R{{\mathbb R}}

\begin{document}
\title{On the type of generalized hypercomplex structures}

\author{Anna Fino}
\address[Anna Fino]{Dipartimento di Matematica ``G. Peano'', Universit\`{a} degli studi di Torino \\
Via Carlo Alberto 10\\
10123 Torino, Italy\\
\& Department of Mathematics and Statistics, Florida International University\\
Miami, FL 33199, United States}
\email{annamaria.fino@unito.it, afino@fiu.edu}

\author{Gueo Grantcharov}
\address[Gueo Grantcharov]{Department of Mathematics and Statistics \\
Florida International University\\
Miami, FL 33199, United States}
\email{grantchg@fiu.edu}

\begin{abstract}  The generalized hypercomplex structures defined within the framework of generalized geometry include hypercomplex and holomorphic symplectic structures as  particular cases. They have a $S^2$-family of generalized complex structures, and in this paper we study  the types of these structures and the corresponding twistor space. We show that there are generalized hypercomplex structures on the $4n$-dimensional tori, which do not contain a structure of maximal (complex) type.
Moreover, we show that the Kodaira-Thurston surface which has a holomorphic symplectic structure, admits also a generalized hypercomplex structure in which all generalized complex structures are of type $1$.

\end{abstract}

\maketitle

\section{Introduction}

  The notion of generalized complex geometry was introduced by N. Hitchin  in the context of generalized Calabi-Yau manifolds with fluxes \cite{Hitchin}. Later, it was developed  by  M. Gualtieri  \cite{Gualtieri}. The main idea is to replace the tangent bundle $TM$ of a  smooth manifold  $M$  with the  generalized tangent bundle, i.e. the  direct sum of the tangent and the cotangent bundle $TM  \oplus  T^*M$,  and the Lie bracket  with the Courant bracket
$$
[X+ \xi,Y + \eta]=[X,Y]+L_X \eta -L_Y \xi  - \frac 12 d(i_X  \eta - i_Y \xi), \quad X,Y \in \Gamma (TM), \xi, \eta \in \Gamma (T^* M).
$$
The Courant bracket admits more symmetries than just diffeomorphisms. The extra symmetries are called $B$-field transformations and are generated by closed $2$-forms on $M$.

Such symmetries lead to a ganeralization of the Courant bracket depending on a closed 3-form $H$ called twisted Courant bracket defined by
$$
[X+ \xi,Y + \eta]_H=[X,Y]+L_X \eta -L_Y \xi  - \frac 12 d(i_X  \eta - i_Y \xi) + i_Xi_Y H , \quad X,Y \in \Gamma (TM), \xi, \eta \in \Gamma (T^* M).
$$
and $B$-field transformations change $H$ to $H+dB$.

A generalized almost complex structures $\mathcal{I}$ is an endomorphism of $TM\oplus T^*M$ with $\mathcal{I}^2 = -{\mbox {Id}}$, compatible with the natural split signature inner product arising from the pairing of vectors and co-vectors. $\mathcal{I}$ is called generalized complex when it is integrable with respect to the (twisted) Courant bracket. Additionally, a choice of a positive definite metric on $TM$ leads to a symmetric isomorphism $\mathcal{G}$ of $TM \oplus T^*M$. When it commutes with $\mathcal{I}$ there is a second generalized almost complex structure $\tilde{\mathcal{I}} = \mathcal{GI}$. If $\tilde{\mathcal{I}}$ is also integrable,
the pair $\mathcal{I}, \tilde{\mathcal{I}}$ is called a (twisted) generalized K\"ahler. The genereralized K\"ahler structures were introduced in \cite{Gualtieri} and are equivalent to  bihermitian structures introduced first by physicists as (2,2)-supersymmetric sigma models \cite{GHR}. Since their appearance, the generalized complex and generalized K\"ahler structures are in the mainstream of research in differential geometry and geometric structures. For some of the recent advances see e.g. \cite{AGJ,  A-C-K, AS, ASU, BCD, B-G, B-G-2, C-N, C-K-W, C-K-W-2, C-M, Mario-Jeff, Goto, Goto2, G-V-GV,Hull-Zabzine, Kob, Mun, Sil, Va, V-SM, Wang, Witten}. An important question in generalized K\"ahler geometry is the construction of compact examples. Apart from the deformation of K\"ahler structures, there are many explicit examples of non-trivial (twisted) generalized K\"ahler structures, e.g.,  \cite{AD,AGG,AG, BF,BFG, BM, CG, DM, FP, FP2, FT}.  In particular, non-K\"ahler examples are given by compact solvmanifolds, in contrast with the case of compact nilmanifolds which cannot admit any invariant generalized K\"ahler structure unless they are tori  \cite{Cavalcanti}.

The generalized  complex structures were also generalized to various other geometric structures on $TM\oplus T^*M$, one of which is generalized hypercomplex. A generalized almost hypercomplex structure  is defined by  a triple  $(\mathcal I, \mathcal J, \mathcal K)$
of anti-commuting generalized almost complex structures such that  $\mathcal K = \mathcal I \mathcal J$. In this case the $S^2$-family $\mathcal{I}_{a,b,c} = a\mathcal{I} + b\mathcal{J} + c\mathcal{K},$  $a^2+b^2+c^2=1$,  is a family of generalized almost complex structures. When the family consists of generalized complex structures, we obtain a generalized hypercomplex structure. By \cite{St2}, it is enough to have two anti-commuting integrable generalized almost complex structure in order to enure that every structure in the family is integrable. In this case the space $Z=M\times S^2$ carries a tautologically defined generalized complex structure and is called the twistor space of $M$ (see Section 4).

The first appearance of generalized hypercomplex structures was in fact as part of a generalized hyperk\"ahler structure in \cite{Brendhaber}. Such structures arise from generalized hypercomplex structures  in the same way as the generalized K\"ahler arise from generalized complex - by adding a compatible positive definite metric on $TM$ which leads to a second generalized almost hypercomplex structure, and when the structures are integrable we have generalized hyperk\"ahler structure. In physics the generalized hyperk\"ahler manifolds correspond to (4,4)-supersymmetric sigma models. For the most recent physics application see \cite{Pap-W, Pap}

Since the only known non-K\"ahler compact examples of generalized hyperk\"ahler manifolds are locally products of compact Lie groups and hyperk\"ahler manifolds, it is natural to study a less restrictive conditions. A prospective candidate is the generalized hypercomplex structure as studied first by M. Stienon \cite{St1}, since both hypercomplex and holomorphic symplectic structures are particular cases of it.

The purpose of this note is to initiate a study of the types of the generalized complex structures in the corresponding hypercomplex $S^2$-family and the generalized complex structure on its twistor space. After the preliminaries we start in Section 3 with two generalized complex structures on a vector space and determine when they are part of a generalized hypercomplex family. Then we discuss the possible types of the structures in the hypercomplex family when the two are chosen among complex and symplectic structures up to a $B$-transform. We extend the condition under which two generalized complex structures define a hypercomplex family from a vector space to a manifold, thus generalizing the result of \cite{St2}. In the proof we use the twistor space and the ideas of D. Kaledin \cite{Kaledin}. This is the first part of Section 4, in the second  part we apply it to characterise the types of local generalized hypercomplex structures when the generalized complex structure on the twistor space is locally either of complex type or of the minimal possible type - one.

The main results of the paper are two examples in Section 5. One is of generalized hypercomplex structure on $4n$-dimensional tori in which the generic structure in the family is of symplectic type and there is no structure of complex type. The second one is a generalized hypercomplex structure on Kodaira-Thurston surface where  every structure in the family is of type $1$. Note that Kodaira-Thurston surface is one of the first examples of non-K\"ahler holomorphic symplectic manifold, so it admits generalized hypercomplex structures of different type.

\section{ Preliminaries on linear generalized complex structures}

 A generalized complex structure on a vector space $V$ is a linear complex structure $\mathcal I$ on $V \oplus V^*$ orthogonal with respect to the pairing
\begin{equation}\label{pairing}
\langle X + \xi, Y + \eta \rangle = \frac 12  (\xi(Y) +\eta (X), \quad X, Y \in V, \, \xi, \eta \in V^*.
\end{equation}
Since $\mathcal I^2 = - {\mbox{Id}}$, the complexification of $V \oplus V^*$ splits as a direct sum of $\pm i$-eingenspaces $L$ and $\overline L$. Moreover, as $\mathcal I$ is orthogonal, we obtain that for $v, w \in L$
$$
\langle v, w \rangle = \langle \mathcal I  v,  \mathcal I w \rangle =  \langle i v, i  w \rangle = - \langle v,  w \rangle.
$$
Therefore  $L$ is a maximal isotropic subspace with respect to the pairing. Conversely, prescribing  such  an $L$ as the  $i$-eigenspace determies a unique generalized complex structure on $V$.   As a consequence  a generalized complex structure on a vector space $V$ is equivalent to a  maximal isotropic subspace  $L$ such that $L  \subset (V \oplus V^*) \otimes \C$ such that $L \cap \overline L = \{ 0 \}$.

\begin{ex}

Examples of generalized complex structure are given by complex and  symplectic structures.
A complex structure $J$ on a real vector space $V$  induces a generalized complex structure on $V$ which can written in matrix form using the splitting  $V\oplus V^*$ as
$$\mathcal I_J = \left(  \begin{array}{cc} -J &0 \\0 & J^*\end{array} \right).
$$
The $i$-eigenspace of $\mathcal I_J$ is $L = V^{0,1} + V^{*{1,0}}$.

A symplectic form $\omega$ on $V$ also induces a generalized complex structure on $V$ by letting
$$ {\mathcal I}_{\omega} = \left(  \begin{array}{cc} 0& -\omega^{-1} \\ \omega & 0\end{array} \right).$$ In this case  the $i$-eigenspace of $\mathcal I_{\omega}$ is  given by
$$
L = \{ X - i  \, \omega (X) \, \mid X \in V \}.
$$
\end{ex}

\smallskip

A real $2$-form $B$ acts naturally on $V \oplus V^*$ by the B-field fransform $$e \mapsto e - B (\pi (e),$$ where $\pi$ is the projection on $V$. If $\mathcal I$ is a generalized complex structure on $V$ whose $i$-eingenspace is $L$,  the image  of $L$ under the action of a $B$-field is given by
$$
L_B = \{ e - i \,  B(e) \, \mid e \in E \}
$$
and it is still isotropic. The associated generalized complex structure $\mathcal I_B$ is called the $B$-field transform of $I$ and is given in matrix form by
$$
\mathcal I_B =  Ad (exp(B)) \mathcal I =  \left(  \begin{array}{cc} 1 &0 \\ -B & 1 \end{array}  \right) \,  \mathcal I  \,  \left(  \begin{array}{cc} 1 &0 \\ B & 1 \end{array}  \right) .
$$
Any two isotropic splittings of $V \oplus V^*$ are related by $B$-field transforms.
More in general, every element $\beta \in \Lambda^2 V$ also acts in a similar way
$$
X + \xi \mapsto X + \xi + \iota_\xi \beta
$$
and  the $\beta$-transform of a generalized complex structure is still a generalized complex structure.

A further characterization of a generalized complex structure on $V$ can be given viewing forms as spinors. The Clifford algebra of $V \oplus V^*$ is defined using the natural  form $\langle \cdot , \cdot \rangle$, i.e. for $v \in V \oplus V^* \subset {\mbox{Cl}}( V \oplus V^*)$ we have $v^2 = \langle v,  v\rangle$. There is a natural action of the Clifford algebra on $\Lambda^{\bullet} V^*$ by
$$
(X + \xi) \cdot\alpha = \iota_X \alpha + \xi \wedge \alpha.
$$
On $\Lambda^{\bullet} V \oplus V^* \subset {\mbox{Cl}}(V \oplus V^*)$ one can define the following bilinear form
$$
(\xi_1, \xi_2) \mapsto ( \sigma(\xi_1) \wedge \xi_2)_{\mbox{top}},
$$
where $\sigma$ is the anti-automorphism of ${\mbox{Cl}}(V \oplus V^*)$ given on decomposable elements by
$$
\sigma (v_1 \cdot v_2  \cdot \cdot \cdot  \cdot \cdot v_k) = v_k  \cdot \cdot \cdot  v_2 \cdot  \cdot \cdot v_1
$$
and ${\mbox{top}}$ denotes the top degree of the form.

Given a form $\rho \in \Lambda^{\bullet} V^* \otimes \C$ of possibly mixed degree one can consider the annihilator
$$
L_{\rho} =  \{ v \in (V \oplus V^*) \otimes \C  \mid v \cdot \rho =0 \}.
$$
Since $ \overline{L_{\rho} } = L_{\overline \rho}$ and
$
0 =v^2 \cdot \rho = \langle v, v  \rangle  \rho, $ for every $v \in L_{\rho}$, the subspace $L_{\rho}$ is isotropic.  If $L_{\rho}$ is maximal, i.e. $\dim_{\C} L_{\rho} = \dim_{\R} V$,  the  element $\rho \in \Lambda^{\bullet} V^*$ is called a pure form.

Given a maximal isotropic subspace $L \subset V \oplus V^*$ one can always find a pure form annihilating it and conversely if two forms annihilate the same isotropic subspace, then they are multiple of each other. Therefore there is a one to one correspondence between maximal isotropic subspaces and lines of pure forms. Algebraically, if $\rho$ is a pure form, then it must be of the form $e^{B + i \omega} \Omega$, where $B$ and $\omega$ are real $2$-forms and $\Omega$ is a decomposable complex form.

 A pure form $e^{B + i \omega} \Omega$ determines a generalized complex structure if and only if  $(\rho, \overline \rho) \neq 0$. This holds only if $V$ is even dimensional $(\dim V = 2n)$, in which case $(\rho, \overline \rho) = \Omega \wedge \overline \Omega \wedge  \omega^{n-k},$ where $k$ is the degree of $\Omega$. Therefore there exist no generalized complex structure on a odd dimensional space  and a generalized complex structure is equivalent to a line $K \subset \Lambda^{\bullet} V^* \otimes \C$ generated by a form $e^{B + i \omega} \Omega$, such that $\Omega$ is  a decomposable complex form of degree $k$, $b$ and $\omega$ are real $2$-forms  and $\Omega \wedge \overline \Omega \wedge \omega^{n -k} \neq 0$. The degree $k$ of the form $\Omega$ is called the type of the generalized complex structure and the line $K$ annihilating $L$ is said the canonical line.

 \begin{rem}  Note that the canonical line  in $\Lambda^{\bullet} V^* \otimes \C$ that gives the generalized complex structure for a complex structure is $\Lambda^{n,0} V^*$ and the line  for a symplectic form $\omega$ is generated by $e^{i \omega}$.

 \end{rem}

 If $\rho$ is a generator of the canonical line of a generalized complex structure $\mathcal I$, $e^B\wedge \rho$ is a generator of a $B$-field transform of $\mathcal I$ and $\iota_{e^\beta} \rho$ is the $\beta$-transform of $\mathcal I$. However, the $\beta$-transform doesn't preserve the Courant bracket.

\smallskip

\section{ Linear generalized hypercomplex structures}

We  introduce now the notion of generalized hypercomplex structure on a vector space $V$.
A generalized hypercomplex structure on $V$  is a pair of anti-commuting complex structures ${\mathcal I,J}$  on  $V\oplus V^*$which are compatible with the symmetrized pairing (\ref{pairing})
Then $K=IJ$ is also a compatible complex structure and the set $$
{\mathcal I}_{a,b,c}= \{ aI+bJ+cK \, \mid \, a^2+b^2+c^2=1 \}$$ forms an $S^2$-family of complex structures compatible with $< \cdot , \cdot >$. We consider hypercomplex structures to be the same if they define the same family.

\smallskip
\
Note that the existence of a generalized hypercomplex structure on the vector space $V$ implies that the dimension of $V$ has to be a  multiple of $4$. This is for instance  a consequence of the results in Section 4 of \cite{Pantilie}.

\smallskip

 The two basic examples of generalized hypercomplex structure are the usual hypercomplex structure on $V$ and  holomorphic symplectic structure, corresponding to  a complex structure and two  symplectic structures $\omega_1, \omega_2$ such that $\omega_1+i \omega_2$ is  a non-degenerate $(2,0)$-form \cite{St1}. The existence of $B$-transform symmetry makes the general case a little more complicated. The first observation is
 \begin{lemma}
 Let ${\mathcal I}_1 \neq \pm{\mathcal I}_2$ be two generalized complex structures on $V\oplus V^*$. Then they determine a generalized hypercomplex structure iff  there is $p, |p|<1$, such that
 \begin{equation}\label{1}
 {\mathcal I}_1{\mathcal I}_2+{\mathcal I}_2{\mathcal I}_1=2p \, {\mbox{Id}}.
 \end{equation}

 \end{lemma}

 \begin{proof}
 The lemma is well known for the usual hypercomplex structures (see e.g. \cite{DGMY}).
 The family is constructed by noticing that $${\mathcal K} = \frac{1}{2\sqrt{1-p^2}}({\mathcal I}_1{\mathcal I}_2-{\mathcal I}_2 {\mathcal I}_1)$$ is a
 generalized complex structure which anti-commutes with ${\mathcal I}_1, {\mathcal I}_2$. Then ${\mathcal I}_1$ and  ${\mathcal K}$ define a
 generalized hypercomplex family which is the same as the one defined by ${\mathcal I}_2$ and ${\mathcal K}$.  The compatibility with  the pairing $\langle   \cdot \, ,  \cdot \, \rangle$ follows from the compatibility of $\mathcal{I}_1$ and $\mathcal{I}_2$, and the equation (\ref{1}).
 \end{proof}

 Now consider the action of $B$-transforms. If  $${\mathcal I}=  \left(  \begin{array}{cc} 0& -\omega^{-1} \\ \omega & 0\end{array} \right)$$
 is a symplectic generalized complex structure, then
 \[
\begin{array}{l}
Ad(exp(B))({\mathcal I})=   \left(  \begin{array}{cc} 1& 0 \\ B & 1\end{array} \right)\left(  \begin{array}{cc} 0& -\omega^{-1} \\ \omega & 0\end{array} \right) \left(  \begin{array}{cc} 1 & 0 \\ -B & 1\end{array} \right)= \left(  \begin{array}{cc} \omega^{-1}B& -\omega^{-1}\\  \omega+B\omega^{-1}B& -B\omega^{-1}\end{array} \right).
\end{array}
\]
 Similarly if ${\mathcal J}= \left(  \begin{array}{cc} J&0\\ 0 & -J^*\end{array} \right)$ is a generalized complex structure of complex type
then
\[
\begin{array}{l}
Ad(exp(B))({\mathcal I})=   \left(  \begin{array}{cc} 1& 0 \\ B & 1\end{array} \right)\left(  \begin{array}{cc} J&0\\ 0 & -J^*\end{array} \right) \left(  \begin{array}{cc} 1& 0 \\ -B & 1\end{array} \right) =  \left(  \begin{array}{cc} J& 0 \\ BJ+J^*B & -J^*\end{array} \right).
\end{array}
 \]

From here we note that $B$-transformation of a structure of complex type like $\mathcal{I}$ 
depends only on the $(2,0)+(0,2)$-part of $B$, since for
$(1,1)$-form $BJ+J^*B=0$. Then the following Lemma is natural:

 \vspace{.1in}

 \begin{lemma}\label{hcxtype}
Let ${\mathcal I}_1, {\mathcal I}_2$ be two generalized complex
structures  on $V  \oplus V^*$ satisfying (\ref{1}) which arise from complex
structures on $V$ up to a $B$-transforms. Then the corresponding
generalized hypercomplex family is given by $$\{{\mathcal
I}_{a,b,c} = Ad(exp(B+B_{a,b,c}))I_{a,b,c} \, \mid \,  (a,b,c)\in S^2\}$$
such that:

(i) $\{I_{a,b,c} \, | \, (a,b,c)\in S^2\}$ are generalized complex
structures which arise from complex structures defining a usual
hypercomplex structure on $V$;

(ii) $B$ is some fixed $2$-form independent of $(a,b,c)$;

(iii) $B_{a,b,c}$ are self-dual 2-forms depending on $(a,b,c)$ -
i.e forms of type $(2,0)+(0,2)$ with respect to every  $I_{a',b',c'}$ for $(a',b',c')\in S^2$.

\end{lemma}

\begin{proof} First note that   a $B$-transformation acts non-trivially on generalized complex structures of complex type only when $B$ is of
type $(2,0)+(0,2)$, so $$Ad(exp(B)){\mathcal I} =
Ad(exp(B^{(2,0)+(0,2)})){\mathcal I},$$ where $B^{(2,0)+(0,2)}$ is
the $(2,0)+(0,2)$-component of $B$. Since $$Ad(exp(B))({\mathcal
I}_1{\mathcal I}_2) = Ad(exp(B))({\mathcal
I}_1)Ad(exp(B))({\mathcal I}_2),$$ we have to check condition
(\ref{1}) for generalized complex structures of the type $$
\begin{array}{l} {\mathcal I}_1= \left(  \begin{array}{cc} J_1&0\\
0 & -J_1^*\end{array} \right)
\end{array}$$ and $$ \begin{array}{l}
{\mathcal I}_2= \left(  \begin{array}{cc} J_2&0\\ BJ_2+J_2^*B & -J_2^*\end{array} \right)
\end{array}.$$ First note that $$BJ_2+J^*_2B=2B_{J_2}^{(2,0)+(0,2)}.$$ A direct calculation shows that (\ref{1}) is satisfied if and only if
$J_1J_2+J_2J_1=2pId|_V$ and
$J^*_1B_{J_2}^{(2,0)+(0,2)}+B_{J_2}^{(2,0)+(0,2)}J_1 =0$. The
second identity is $(B_{J_2}^{(2,0)+(0,2)})_{J_1}^{(1,1)}=0$. From
above, we can consider $B=B_{J_2}^{(2,0)+(0,2)}$, so then $B\in
\Lambda^{(2,0)+(0,2)}_{J_2} \cap  \Lambda^{(2,0)+(0,2)}_{J_1}$.
Then one can easily check that $B$ is of type $(2,0)+(0,2)$ with
respect to any structure in the family determined by $J_1,J_2$ and
conversely - every structure in a hypercomplex family can be
changed by a $B$-transform in such way and still the family
defines generalized hypercomplex  structure. This proves the
Lemma.
\end{proof}

We call such generalized hypercomplex structures {\it $B$-twisted hypercomplex}, since the $B$-transformations depend on the point $(a,b,c)\in S^2$.

\vspace{.1in}

Next we consider the case when two generalized complex structures
of symplectic type satisfy (\ref{1}). To compare with the case of
complex structures we consider first the types of the structures
in the $S^2$-family of generalized complex structures, defined by
a holomorphic symplectic structure. Let $${\mathcal I}=\left(
\begin{array}{cc} J&0\\ 0 & -J^*\end{array} \right), \quad
{\mathcal J}=\left(  \begin{array}{cc} 0& -\omega^{-1}\\ \omega &
0\end{array} \right)
$$ for a complex structure $J$ and a symplectic structure $\omega$ such
that $J^*\omega=\omega J$. Then
 $${\mathcal K}={\mathcal IJ} = \left(  \begin{array}{cc} 0& -J\omega^{-1}\\ -J^* \omega  & 0\end{array} \right).$$ From $J^*\omega=\omega
 J$ we get $J^*=\omega J \omega^{-1}$ and $J=\omega^{-1}J^*\omega$. We check directly that
$$
\begin{array}{l}
(b \omega + c \omega J) (- b \omega^{-1} + c \omega^{-1} J^*)= -(b^2+c^2) {\mbox {Id}},\\[3 pt]
(b \omega J - c \omega) (- b \omega^{-1} + c \omega^{-1} J^*) = -
(b^2 + c^2) J^*,\\[3pt]
(- b \omega^{-1} + c \omega^{-1} J^*)(b \omega J - c \omega)=
-(b^2+c^2)J.
\end{array}
$$

Using this we see that
 \[
\begin{array}{l}
{\mathcal I}_{a,b,c}=a{\mathcal I}+b{\mathcal J}+c{\mathcal K} = \left(  \begin{array}{cc} aJ& -b\omega^{-1}+c\omega^{-1} J^*\\  b\omega+c\omega J & -aJ^*\end{array} \right) = \\[4 pt]
\left(  \begin{array}{cc} 1& 0\\ \frac{a}{b^2+c^2}(b\omega J-c\omega) & 1\end{array} \right)\left(  \begin{array}{cc} 0& -b\omega^{-1}+c\omega^{-1} J^* \\ \frac{1}{b^2+c^2}(b\omega+c\omega J) & 0\end{array} \right)\left(  \begin{array}{cc} 1& 0\\ -\frac{a}{{b^2+c^2}}(b\omega J-c\omega) & 1\end{array} \right)
\end{array}
\]
and  so every structure in the family except $\pm {\mathcal I}$ is
of symplectic type.

Now we check when $\mathcal{I}_1, \mathcal{I}_2$ of symplectic type define a holomorphic symplectic structure up to a $B$-transform. 

\begin{lemma}\label{hypersympl}
Let ${\mathcal I}_1, {\mathcal I}_2$ be two generalized complex structures of symplectic type  on $V \oplus  V^*$ satisfying (\ref{1}), such that
\[ \begin{array}{l}
{\mathcal I}_1= \left(  \begin{array}{cc} 0& -\omega_1^{-1}\\ \omega_1 & 0\end{array} \right)
\end{array}\]
 and
 \[ \begin{array}{l}
{\mathcal I}_2= \left(  \begin{array}{cc} \omega_2^{-1}B& -\omega_2^{-1}\\  \omega_2+B\omega_2^{-1}B& -B\omega_2^{-1}\end{array} \right).
\end{array}\]
Let $A=B\omega_2^{-1}$ and $D=\omega_1\omega_2^{-1}+pId$. Then
\begin{equation} \label{condequivAD}
A^2+D^2=(p^2-1)Id\hspace{.3in} AD=DA.
\end{equation}
Conversely,  for  given symplectic forms $\omega_1, \omega_2$,
a closed form $B$, and a real number $p$ such that the endomorphisms
$A$ and $D$ defined as above satisfy (\ref{condequivAD}),  the
corresponding ${\mathcal I}_1$ and ${\mathcal I}_2$ define a
generalized hypercomplex structure.

\end{lemma}

\begin{proof} The proof is a direct calculation:
\[
{\mathcal I}_1{\mathcal I}_2 = \left(  \begin{array}{cc} -\omega_1^{-1}\omega_2 -\omega_1^{-1}B\omega_2^{-1}B& \omega_1^{-1}B\omega_2^{-1}\\  \omega_1\omega_2^{-1}B&- \omega_1\omega_2^{-1}\end{array} \right)
\]
and
\[
{\mathcal I}_2{\mathcal I}_1 = \left(  \begin{array}{cc} -\omega_2^{-1}\omega_1& - \omega_2^{-1}B\omega_1^{-1}\\  -B\omega_2^{-1}\omega_1& -\omega_2\omega_1^{-1}-B\omega_2^{-1}B\omega_1^{-1} \end{array} \right).
\]

Then \eqref{1}  is equivalent to the system
\begin{equation} \label{systemAD}
\left \{ \begin{array}{l}
\omega_2^{-1}\omega_1+\omega_1^{-1}\omega_2+\omega_1^{-1}B\omega_2^{-1}B= -2p \,  {\mbox{Id}},\\
\omega_1\omega_2^{-1}+\omega_2\omega_1^{-1}+B\omega_2^{-1}B\omega_1^{-1}= -2p \, {\mbox{Id}},\\
\omega_1^{-1}B\omega_2^{-1}-\omega_2^{-1}B\omega_1^{-1}=0,\\
\omega_1\omega_2^{-1}B-B\omega_2^{-1}\omega_1=0.
\end{array} \right.
\end{equation}
Since $$\omega_1
(\omega_2^{-1}\omega_1+\omega_1^{-1}\omega_2+\omega_1^{-1}B\omega_2^{-1}B)\omega_1^{-1}
=\omega_1\omega_2^{-1}+\omega_2\omega_1^{-1}+B\omega_2^{-1}B\omega_1^{-1}$$
and
$$\omega_1(\omega_1^{-1}B\omega_2^{-1}-\omega_2^{-1}B\omega_1^{-1})\omega_1
=\omega_1\omega_2^{-1}B-B\omega_2^{-1}\omega_1, $$ the third and the
four (resp. the first and second) equations in \eqref{systemAD}
are equivalent. So it is sufficient to show that the second and
third equation in \eqref{systemAD} are equivalent to the equations
\eqref{condequivAD}.

 By the definition of $D$ and the second equation  in \eqref{systemAD} we have
\begin{equation} \label{expressD} D = \omega_1\omega_2^{-1}+p \, {\mbox{Id}} =
- \omega_2 \omega_1^{-1} - B \omega_2^{-1} B \omega_1^{-1} - p \,  {\mbox{Id}}.
\end{equation}
Thus, by \eqref{expressD}    and the definition of $A$ we get
$$
AD = B\omega_2^{-1}(- \omega_2 \omega_1^{-1} - B \omega_2^{-1} B
\omega_1^{-1} - p Id) = - B \omega_1^{-1} - B \omega_2^{-1} B
\omega_2^{-1} B \omega_1^{-1} - p B \omega_2^{-1}.
$$
Similarly we have
$$
DA = - \omega_2 \omega_1^{-1} B \omega_2^{-1} - B \omega_1^{-1}  B \omega_2^{-1} B \omega_2^{-1} - p B \omega_2^{-1}.
$$
Multiplying the third equation in \eqref{systemAD} by $\omega_2$
to the left gives
$$
\omega_2 \omega_1^{-1} B \omega_2^{-1}  = B \omega_1^{-1}
$$
and multiplying by $B$ to the left gives
$B\omega_1^{-1}B\omega_2^{-1} = B\omega_2^{-1}B\omega_1^{-1}$, so
$$
B\omega_1^{-1} B \omega_2^{-1} B \omega_2^{-1} = B \omega_2^{-1} B
\omega_1^{-1} B \omega_2^{-1} = B \omega_2^{-1} B \omega_2^{-1} B
\omega_1^{-1}
$$

This proves $AD = DA$. For the second relation between $A$ and $D$
we obtain
$$
A^2 + D^2 = (B \omega_2^{-1})^2 + (\omega_1 \omega_2^{-1} + p \,  {\mbox{Id}})^2 =  B \omega_2^{-1}B \omega_2^{-1} + \omega_1 \omega_2^{-1}  \omega_1 \omega_2^{-1} + 2 p \omega_1 \omega_2^{-1}  + p^2 \, {\mbox{Id}}.
$$
By the second equation   in \eqref{systemAD} $ 2 p  \, {\mbox{Id}} = - \omega_1
\omega_2^{-1} - \omega_2 \omega_1^{-1} - B \omega_2^{-1} B
\omega_1^{-1} $ and multiplying by $\omega_1\omega_2^{-1}$
$$
2p \, \omega_1\omega_2^{-1} = -\omega_1 \omega_2^{-1}  \omega_1
\omega_2^{-1}- {\mbox{Id}} - B \omega_2^{-1}B \omega_2^{-1}
$$
so
$$
B \omega_2^{-1} B \omega_2^{-1} + \omega_1 \omega_2^{-1}  \omega_1 \omega_2^{-1} + 2 p  \, \omega_1 \omega_2^{-1}  = - {\mbox{Id}}.
$$
Therefore $A^2 + D^2 = (p^2 - 1) {\mbox{Id}}$. The converse follows
directly.

\end{proof}

From here  we see that two  generalized complex
structures  of symplectic type satisfying \eqref{1} will define a holomorphic symplectic structure when
they are not twisted by a $B$-transform. However, in general this
is not true. We call the generalized hypercomplex structures which
are $B$-transformations with $B$  not depending on $(a,b,c)\in
S^2$ of the structures in Lemma \ref{hypersympl} {\it $B$-transformed
hypersymplectic}.

\smallskip

 Now we want to see when to see when two generalized complex structures as in \eqref{1} define a holomorphic symplectic structure:
 \begin{lemma}\label{holosymplectic}  Suppose the generalized complex structures $\mathcal{I}_1, \mathcal{I}_2$ are as in Lemma \ref{hypersympl} satisfying \eqref{condequivAD}. Then the generalized hypercomplex structure they define contains a structure of maximal type (i.e.  a  $B$-transformed complex structure) if and only if there are real numbers  $a,b,c$ such that
 \begin{equation}\label{2}
 cB = a\omega_2+b\omega_1.
 \end{equation}
In case $c\neq 0$ the equations \eqref{systemAD} is equivalent to the single equation:
\begin{equation}\label{holosymp}
(\omega_2^{-1}\omega_1 + \gamma Id)^2 = \theta  {\mbox{Id}},
\end{equation}
for some  real numbers $\gamma$ and  $\theta<0$.
\end{lemma}

\begin{proof}
If we take $\mathcal{K} = \frac{1}{2\sqrt{1-p^2}}[\mathcal{I}_1, \mathcal{I}_2]$, then the hypercomplex family  defined by $\mathcal I_1$ and $\mathcal I_2$  is in the span of $\mathcal{I}_1, \mathcal{I}_2, \mathcal{K}$, so by the assumption of the Lemma there are $a,b,c$ such that the type of $\mathcal{I}_{a,b,c} = a\mathcal{I}_1+b\mathcal{I}_2+c\mathcal{K}$ is maximal. A direct computation of the Poisson vector of $\mathcal{I}_{a,b,c}$ gives
\[
-a\omega^{-1}_1-b\omega_2^{-1} +\frac{c}{\sqrt{1-p^2}}\omega_1^{-1} B \omega^{-1}_2 = 0.
\]
Multiplying by $\omega_1$ to the left and $\omega_2$ to the right we obtain \eqref{2} with $c$ replacing $\frac{c}{\sqrt{1-p^2}}$.
From here we see that the equation $AD = DA$ in the system \eqref{condequivAD} is satisfied and the first equation in \eqref{condequivAD} is
\[
(B\omega_2^{-1})^2+(\omega_1\omega_2^{-1}+p Id)^2 = (p^2-1) {\mbox{Id}}
\]
If $c\neq 0$, the equation \eqref{2} becomes
\[
B=\alpha\omega_1+\beta\omega_2
\]
and substituting it in the first equation of \eqref{condequivAD} we obtain
\[
(\alpha^2+1)(\omega_1\omega_2^{-1})^2+2(\alpha\beta+p)\omega_1\omega_2^{-1} = -(1+\beta^2) {\mbox{Id}}
\]
which by completing the square gives \eqref{holosymp} with $\gamma=\frac{\alpha\beta+p}{\alpha^2+1}, \theta = (\frac{\alpha\beta+p}{\alpha^2+1})^2-\frac{1+\beta^2}{1+\alpha^2}$. The fact that $\theta<0$ follows from $-1<p<1$.
\end{proof}

Now we consider the case that one of the structures in the family has maximal type, i.e. is equivalent up to a B-transformed complex structure. By fixing the complex structure, we see that a $B$-transformed holomorphic symplectic structure is
described by the following:

\begin{lemma}\label{B-holosympl}
Let ${\mathcal I}_1, {\mathcal I}_2$ be two generalized complex
structures on $V \oplus V^*$  satisfying (\ref{1}), such that
\[ \begin{array}{l}
{\mathcal I}_1= \left(  \begin{array}{cc} J& 0\\
0 & -J^*\end{array} \right)
\end{array}\]
 and
 \[ \begin{array}{l}
{\mathcal I}_2= \left(  \begin{array}{cc} \omega^{-1}B&
-\omega^{-1}\\  \omega+B\omega^{-1}B& -B\omega^{-1}\end{array}
\right)
\end{array}\]

Then the generalized hypercomplex structure defined by ${\mathcal
I}_1$ and ${\mathcal I}_2$  is a $B$-transformed  holomorphic symplectic
structure.
\end{lemma}

\begin{proof}
Again we have:
\[
{\mathcal I}_1{\mathcal I}_2 = \left(  \begin{array}{cc}
J\omega^{-1}B& -J \omega^{-1}\\
-J^*\omega-J^*B\omega^{-1}B& J^*B\omega^{-1}\end{array} \right)
\]
and
\[
{\mathcal I}_2{\mathcal I}_1 = \left(  \begin{array}{cc}
-\omega^{-1}BJ& - \omega^{-1}J^*\\
\omega J +B\omega^{-1}BJ& B\omega^{-1}J^*
\end{array} \right).
\]
So (\ref{1}) is equivalent to:
\begin{equation} \label{systemholosympl}
\left \{ \begin{array}{l}
\omega J=J^*\omega \\
\omega J-J^*\omega +B\omega^{-1}BJ-J^*B\omega^{-1}B = 0\\
J\omega^{-1}B+\omega^{-1}BJ=2pId\\
J^*B\omega^{-1}+B\omega^{-1}J^*=2pId.
\end{array} \right.
\end{equation}

Similarly as before using the first equation we see that the third
and fourth equations are equivalent to $J^*B+BJ=2p \omega$.

Now consider the structure ${\mathcal
K}=\frac{1}{\sqrt{1-p^2}}[{\mathcal I}_1, {\mathcal I}_2]$. By
direct calculation using (\ref{systemholosympl}) we see that
${\mbox{rank}} ({\mathcal K})=  {\mbox{rank}}(\omega^{-1}J^*)$, which is maximal. Then
${\mathcal K}$ is a $B$-transform of some symplectic structure and
the generalized hypercomplex structure defined by ${\mathcal I}_1$
and ${\mathcal I}_2$ is the same as the one defined by ${\mathcal
I}_1$ and ${\mathcal K}$. So we can replace ${\mathcal I}_2$ by
${\mathcal K}$ in the Lemma and consider only the case $p=0$. But
then from above $J^*B + BJ =0$ and $Ad(\exp (B)){\mathcal I}_1 =
{\mathcal I}_1$. From here it follows that the generalized
hypercomplex structure defined by ${\mathcal I}_1, {\mathcal I}_2$
arises from a holomorphic symplectic structure which is
transformed by $B$.

\end{proof}

\vspace{.2in}

\section{Generalized hypercomplex manifolds and their twistor spaces}

Recall that the  direct sum  $TM  \oplus  T^*M$  of the tangent and the cotangent bundle of a manifold  $M$ carries a Courant bracket
$$
[X+ \xi,Y + \eta]=[X,Y]+L_X \eta -L_Y \xi  - \frac 12 d(i_X  \eta - i_Y \xi), \quad X,Y \in \Gamma (TM), \xi, \eta \in \Gamma (T^* M).
$$
The Courant bracket admits more symmetries than just diffeomorphisms. The extra symmetries are called $B$-field transformations and are generated by closed $2$-forms on $M$. They correspond to $B$-transformations of $T_pM \oplus T^*_pM$ at each point $p$ in $M$. Note that the Courant bracket we define is skew-symmetric and doesn't satisfy the Jacobi identity, but it is possible to define a bracket, which satisfies the Jacobi identity and is not skew-symmetric.

A {\em generalized almost complex structure} on $M$  is an endomorphism $\mathcal{I}$ of $TM \oplus T^*M$ of square $- {\mbox{Id}}$  and compatible with the canonical pairing. The $\mathcal{I}$ splits the complexification $(TM \oplus T^*M)\otimes\mathbb{C}$ into complex $\pm\sqrt{-1}$-eigenbundles $E^{(1,0)}$ and $E^{(0,1)}$ of type $(1,0)$ and $(0,1)$.
A generalized almost complex structure is called integrable if one (and hence both) of these eigenbundles are closed under the complexified Courant bracket. It is equivalent also to vanishing of the generalized version of the Nijenhuis tensor, which is again a real tensor, defined via the Courant bracket. An integrable generalized almost complex structure is called {\em generalized complex structure}.

Now let ${\mathcal I, J}$ be anti-commuting generalized complex structures. By the result in \cite{St2}, ${\mathcal K}={\mathcal IJ}=-{\mathcal JI}$ is also integrable and every structure in  the family $\{{\mathcal J}_{\bf a} \, | \, {\bf a}\in S^2\}$ is also integrable.

\begin{definition}
The  {\em twistor space} of a generalized almost hypercomplex manifold $M$ is the smooth manifold $Z=M\times S^2$ endowed
with the generalized almost complex structure ${\mathbb I}$ defined at $T_{(p,{\bf a})}Z \oplus T^*_{(p,{\bf a})}Z$ as
\[
{\mathbb I}|_{(p,{\bf a})} = {\mathcal I}_{\bf a}|_{T_pM \oplus T^*_pM}\oplus {\mathcal J}|_{T_{\bf a}S^2 \oplus T^*_{\bf a}S^2},
\]
where ${\mathcal J}$  is the generalized complex structure on $S^2$ determined by the canonical complex structure via $S^2\cong{\mathbb C}P^1$.
\end{definition}

The definition is \lq \lq half" of the definition in  \cite{Brendhaber}  of a twistor space of generalized hyperk\"ahler manifold. Note that we can use the Fubini-Studi symplectic form on $S^2$ for ${\mathcal J}$ in the definition and the following Theorem is still valid:
\begin{theorem}
When $M$ is generalized hypercomplex, the structure ${\mathbb I}$ on $Z$ preserves the generalized tangent bundles of $M$ and $S^2$ and is integrable, i.e. it defines a generalized complex structure.
\end{theorem}

The proof is the same is the proof for the usual twistor space for hypercomplex manifolds.  Take sections $X^h, Y^h$ of  $T^cM \oplus (T^*)^cM$ and $U,V$ of $T^cS^2 \oplus (T^*)^cS^2$ of the complexified generalized tangent bundles which are of type $(1,0)$ with respect to ${\mathbb I}$ and such that $$X^h|_{(p, {\bf a})}=(X,\alpha)|p-i{\mathbb I}_{\bf a}(X,\alpha)|_p, Y^h|_{(p, {\bf a})}=(Y,\beta)|p-i{\mathbb I}_{\bf a}(Y,\beta)|_p$$ for some sections $(X,\alpha), (Y,\beta)$ of $TM \oplus T^*M$. Clearly such sections span the $(1,0)$-subspace of $T^cZ \oplus (T^*)^cZ$ at each point $(p,{\bf a})$. Then one can check directly that $[X^h,Y^h]$ and $[U,V]$ are in the same $(1,0)$-subspace while $[X^h,U]=[X^h,V]=[Y^h,U]=[Y^h,V]=0$.

\vspace{.1in}

It is useful for the twistor theory to identify $S^2$ with ${\mathbb C}P^1  \cong  \C \cup \{ \infty \}$ via the stereographic projection: $$St:\lambda \in  \C \cup \{ \infty \} \rightarrow  \left (\frac{1-|\lambda|^2}{1+|\lambda|^2}, \frac{i(\overline{\lambda}-\lambda)}{1+|\lambda|^2}, \frac{\lambda+\overline{\lambda}}{1+|\lambda|^2} \right)\in S^2.$$
If ${\mathcal I}_{\lambda}$ denotes the structure corresponding to
$\lambda$ via the inverse map $St^{-1}(a,b,c)$, then ${\mathcal
I}={\mathcal I}_0$ and ${\mathcal J}={\mathcal I}_{\sqrt{-1}}$.
Moreover there is a basis $Z_1,W_1,...,Z_n,W_n$ of $(1,0)$-vectors
with respect to $I$ such that $J(Z_i)=-\overline{W_i},
J(W_i)=\overline{Z_i}$. Then the structure $I_{\lambda}$ has
$(1,0)$-space given by $$Span\{ Z_i+\lambda\overline{W_i},
W_i-\lambda\overline{Z_i} \, | \,  i = 1,2...,n\}.$$ Note that ${\mathcal
I}_{\lambda}$ exhaust all structures in the family except
$-{\mathcal I}$, which corresponds to the point at infinity of the
complex plane with parameter $\lambda$.

\vspace{.1in}

We want to compare the type of the generalized complex structure on $Z$ with the types of the generalized hypercomplex structures consider in Section 3. To this end we need the following characterization of generalized hypercomplex structures, which generalizes the result in \cite{St2} mentioned above:

\begin{theorem}
Two generalized complex structures $\mathcal{I}_1$ and $\mathcal{I}_2$ are part of a generalized hypercompllex family if and only if they satisfy
\begin{equation}
 {\mathcal I}_1{\mathcal I}_2+{\mathcal I}_2{\mathcal I}_1=2p \, {\mbox{Id}}
 \end{equation}
for a  real number $p$ with $|p|<1$.

\end{theorem}

\begin{proof}
The analog for the usual hypercomplex structures is proven by D. Kaledin \cite{Kaledin} and here we use his approach. We first notice that by Lemma 2.3 we obtain a generalized almost complex family containing $\mathcal{I}_1$ and $\mathcal{I}_2$.

First consider the linear hypercomplex structure on $V\oplus V^*$ for $V=T_pM$ for some $p$, defined by anticommuting $\mathcal{I}, \mathcal{J}$ and (1,0)-vectors $X,Y$  for $\mathcal{I}$ such that $\mathcal{J}(X)=\overline{Y}$, so $\mathcal{J}(Y) = -\overline{X}$. Since $\mathcal{J}$ maps the  $(1,0)$-space for $\mathcal{I}$ isomorphically to the (0,1)-space, there is a basis of this (1,0)-space
 $X_1,Y_1,X_2,Y_2,...X_n,Y_n$ with  $\mathcal{J}(X_i)=\overline{Y_i}, \mathcal{J}(Y_i) = -\overline{X_i}$. Using the stereographic projection, on can consider $\mathcal{I}_{a,b,c} = \mathcal{I}_{\lambda}$ for $\lambda \in \mathbb{CP}^1$ being the image of $(a,b,c)\in S^2$. Then as mentioned above one can see that for $\lambda\neq\infty$ a (1,0)-basis for  $\mathcal{I}_{\lambda}$ is given by
\[
 X_i^{\lambda} = X_i-\lambda \overline{Y_i},\hspace{.4in}
 Y_i^{\lambda} = Y_i + \lambda \overline{X_i}.
 \]
Here $\lambda=i$ corresponds to $\mathcal{J}$. Moreover $\lambda = \infty$ corresponds to $-\mathcal{I}$.

\vspace{.1in}

Now consider a manifold with generalized complex structures $\mathcal{I}_1, \mathcal{I}_2$ satisfying $\mathcal{I}_1\mathcal{I}_2 +\mathcal{I}_2\mathcal{I}_1 = 2p Id$ where $|p|<1$ is constant function. Then $$\mathcal{I}_1, \, (\mathcal{I}_2+p\mathcal{I}_1)/\sqrt{1-p^2}, \,  [\mathcal{I}_1, \mathcal{I}_2]/2\sqrt{1-p^2}$$ is a triple of anti-commuting generalized almost complex structures in general. We'll show that they are actually integrable. To this end consider the Nijenhuis tensor of ${\mathcal I}_{\lambda}$ coupled with the pairing  $< \cdot , \cdot >$ on  $ T M\oplus T^* M$ for the structure $\mathcal{I}_{\lambda}$:
\[ N_{\lambda}(U,V,W) = <[U,V]-\mathcal{I}_{\lambda}[\mathcal{I}_{\lambda}U,V]-\mathcal{I}_{\lambda}[U,\mathcal{I}_{\lambda}V]-[\mathcal{I}_{\lambda}U,\mathcal{I}_{\lambda}V], W>,
\]
where $[ \cdot , \cdot ]$ is the Courant bracket. The fact that this is a 3-form on $TM \oplus T^* M$ is proved in \cite{St2}.
It is clear that $\mathcal{I}_{\lambda}$ is integrable iff $N_{\lambda}=0$. Moreover $N_{\lambda}=0$ iff $N_{\lambda}(U_{\lambda},
V_{\lambda}, W_{\lambda})=0$ for every (1,0)-vector fields $ U_{\lambda},
V_{\lambda}, W_{\lambda}$ with respect to $\mathcal{I}_{\lambda}$. An important fact is that $$N_{\lambda}(U_{\lambda},
V_{\lambda}, W_{\lambda})= <[U_{\lambda},
V_{\lambda}], W_{\lambda}>.$$

\vspace{.1in}

Since $N_{\lambda}$ is a tensor, to prove Lemma 1.1, it is enough to fix an arbitrary vector fields $ U_{\lambda},
V_{\lambda}, W_{\lambda}$ among the basis $X_i^{\lambda}, Y_i^{\lambda}$ above. Then

 $$N_{\lambda}(U_{\lambda},
V_{\lambda}, W_{\lambda}) = \lambda^3 N_3 + \lambda^2 N_2 + \lambda N_1 + N_0$$ where $N_i$ for each $i=0,1,2,3$ are values of a function on $M$ of the type
$$ p \rightarrow <[X, Y] \vert_p, Z\vert_p>,$$
where $X,Y,Z$ are among  $X_j, Y_j, \overline{X_j},\overline{Y_j}$, $j=1,2,...,n$.
Now as in D. Kaledin's argument, we can fix the point $p$ and notice, that
 $N_{\lambda}(U_{\lambda},
V_{\lambda}, W_{\lambda})$  becomes a polynomial of degree 3 in the variable  $\lambda$,  with 4 different roots corresponding to $\pm\mathcal{I}_1,\pm\mathcal{I}_2$, one of them at $\infty$. But this means that it vanishes identically. Since we choose $p$ and
 $U_{\lambda},
V_{\lambda}, W_{\lambda}$ arbitrary, the theorem follows.

\end{proof}

\vspace{.1in}

By the result in \cite{Gualtieri}  at generic points  $(Z,{\mathbb I})$ is a product of complex and symplectic submanifolds and the dimension of
the complex submanifold is also the type of ${\mathbb I}$. From the definition it follows that the type of ${\mathbb I}$ is at least one. The next observation follows from Lemma \ref{hcxtype} and Lemma \ref{hypersympl}:
\begin{theorem}
If the real dimension of $M$ is $4n$ and the type of ${\mathbb I}$ on $Z$ is $2n+1$ in some open set, then the family ${\mathbb I}|{\bf a}$ locally is a $B$-twisted hypercomplex structure as in Lemma \ref{hcxtype}. If the type of ${\mathbb I}$ is one at a point, then the family ${\mathbb I}|{\bf a}$, locally is of the type described in Lemma \ref{hypersympl}. If in addition for two of the structures in the $S^2$-family satisfy \eqref{2} and \eqref{holosymp}, then the structure is B-transformed holomorphic symplectic.
\end{theorem}

\begin{proof}
We notice, that the conditions imply that the type is constant in an open set in both cases, since 1 is the minimal rank of $\mathbb{I}$ and being of minimal rank is an open condition. Such open set is of the form $U\times V$ with $U\subset M$ and $V\subset S^2$. Then in particular there are  two generalized complex structures $\mathcal{I}_1$ and $\mathcal{I}_2$ which satisfy the conditions of Lemmas \ref{hcxtype} or \ref{hypersympl} for some $p$ close enough to 1 since they will correspond to close $\lambda$ values in $V$.
\end{proof}

The question about description of the generalized hypercomplex structure in case the type is between $2$  and $2n$ is open at the moment.

\section{Examples on tori and Kodaira-Thurston surface}

\begin{ex}\label{ex1}  Let $e^1,f^1,...,e^{2n},f^{2n}$ be the standard basis of $({\mathbb R}^{4n})^*$ and
$$\omega_2 =\sum_{i} (e^{2i-1}\wedge f^{2i-1}-e^{2i}\wedge f^{2i})$$
 be the real part of the standard holomorphic symplectic form on $({\mathbb R}^{4n})^*$ and
 $$\omega_1 = \sum_{i} \lambda_i(e^{2i-1}\wedge f^{2i}+e^{2i}\wedge f^{2i-1})$$ and
$$B=\sum_{i} \mu_i(e^{2i-1}\wedge f^{2i}+e^{2i}\wedge f^{2i-1})$$ be \lq \lq diagonal deformations" of the imaginary part such that
 $\lambda_i^2+\mu_i^2=1$. Then the conditions of the Lemma \ref{hypersympl} are satisfied with $p=0$ and
 $$A=   {\mbox{diag}} \left ( \left(  \begin{array}{cccc} 0&0&- \mu_i&0\\ 0&0&0&-\mu_i\\
 \mu_i&0&0&0\\ 0&\mu_i&0&0\end{array} \right) \right ),
\quad D =  {\mbox{diag}}  \left (\left(  \begin{array}{cccc} 0&0&- \lambda_i&0\\ 0&0&0&-\lambda_i\\
 \lambda_i&0&0&0\\ 0&\lambda_i&0&0\end{array} \right) \right ).
 $$

From Lemma \ref{hypersympl} we obtain two generalised complex structures $\mathcal{I}_1, \mathcal{I}_2$ which define a generalized hypercomplex structure. However, the conditions of Lemma \ref{holosymplectic} lead to the fact that this structure  is  a $B$-transformed holomorphic symplectic  structure only if $\lambda_i=\lambda$ and $\mu_i=\mu$ for some $\lambda$ and $\mu$. On another side we see that the rank of the generalized complex structures in the family defined by this generalized hypercomplex structure could jump, since the rank of $a\omega_2+b\omega_1+cB$ could change for different $a,b,c$ depending on particular $\lambda_i, \mu_i$. However, no structure is of maximal type unless  $\lambda_i=\lambda$ and $\mu_i=\mu$ for some $\lambda$ and $\mu$.

\vspace{.2in}

\end{ex}

The above example of a linear generalized hypercomplex structure defines a generalized hypercomplex structure on $4n$-dimensional tori. To see this we just notice the basis $(e_1,f_1,...,e_{2n},f_{2n})$ could be selected among the standard vector fields and 1-forms $\frac{\partial}{\partial x_i}, dx_i, \frac{\partial}{\partial y_i}, dy_i$ for coordinates $(x_1,y_1,...,x_{2n},y_{2n})$ of the universal cover. Since all Courant brackets among them vanish, the generalized almost hypercomplex structure defined in Example \ref{ex1} is in fact generalized hypercomplex.  We can summarize these observations in:
\begin{prop}
The $4n$-dimensional tori admit a generalized hypercomplex structure in which the generic member of the family is of symplectic type, but the family contains no structure of complex type, hence is not $B$-transformed holomorphic symplectic.
\end{prop}

Next we provide an example of compact $4$-manifold admitting a generalized hypercomplex structure which is neither  a hypercomplex nor  a holomorphic symplectic structure. The compact $4$-manifold is  the Kodaira-Thurston surface, i.e. the nilmanifold given by the compact quotient of $H_3 \times \R$ by a lattice, where $H_3$ is the real  Heisenberg $3$-dimensional Lie group.

In general the cotangent bundle  $T^* G$  of a Lie group $G$  with Lie algebra $\mathfrak g$  has a canonical Lie group
structure induced by the coadjoint action of  the Lie group $G$  on  the Lie algebra $\mathfrak g$   and  it has also a canonical bi-invariant
neutral metric. Hermitian structures on $T^*G$  such that left
translations are holomorphic isometries are given by  almost complex structures  of $\mathfrak g \oplus \mathfrak  g^*$
 which are orthogonal with respect to the neutral inner product
 $$
h( (X, \alpha), (Y, \beta)) = \frac{1}{2} (\beta (X) + \alpha (Y))
 $$
and satisfying  $N_J =0$, where  $N_J$ is  Nihenhuis tensor defined  with respect to the Lie bracket:
\begin{equation} \label{bracketcotang} [(X, \alpha), (Y, \beta)] = ([X, Y], - \beta \circ ad_X + \alpha \circ ad_Y).\end{equation}

 By  Proposition 3.1  in \cite{ABDF}  there is  a one-to-one correspondence between left-invariant  hermitian structures  $(J, h)$ on $T^* G$, where $h$ is  the
standard neutral metric on $T^* G$
and left-invariant generalized complex structures on $G$.

Consider on  $H_3 \times \R$ the  structure equations
$$ \begin{array}{l}
d e^j =0,  \, j = 0,1,2,\\[4pt]
d e^3 = - e^1 \wedge e^2.
\end{array}
$$

The cotangent  bundle  $T^* (H_3 \times \R)$ has, with respect to the Lie bracket \eqref{bracketcotang}, structure equations:
$$
\begin{array}{l}
d E^j =0, \,  j = 1,2,3,5, 8,\\[4pt]
d E^4 = - E^2 \wedge E^3,\\[4pt]
d E^6 =  -E^3 \wedge E^8,\\[4pt]
d E^7 = E^2 \wedge E^8,
\end{array}
$$
where $(E^j)$ is the dual basis of $(E_j)$ with
$$
\begin{array}{l}
E_1 = (e_0, 0),  \, E_2 = (e_1, 0),  \, E_3 = (e_2, 0),  \, E_4 = (e_3, 0),\\[3pt]
 E_5 =(e^0, 0),  \, E_6 =(e^1, 0),  \, E_7 =(e^2, 0),  \, E_8 =(e^3, 0)
 \end{array}
$$
and $(e^i)$ is the dual basis of  $(e_i)$.

The Lie group $H_3 \times \R$ has a  pair of anti-commuting   left-invariant generalized complex structures ${\mathcal I, \mathcal J}$ of type $1$, which are given respectively by the  left-invariant complex structures on $T^* (H_3 \times \R)$ defined  by the matrices
$$
I=  \left ( \begin{array}{cccccccc}  0&  -1& 0&  0&  0&  0& 0&  0\\[3pt]
 1&  0& 0&  0&  0&  0&  0& 0\\[3pt]
 0& 0& 0& 0& 0&0&0& 1\\[3pt]
 0& 0& 0& 0&0&0&-1& 0\\[3pt]
0& 0& 0& 0&0& -1& 0&0\\[3pt]
0& 0&0& 0& 1&0& 0& 0, \\[3pt]
 0&0& 0&1&0& 0&0&0\\[3pt]
 0&0& -1&0&0& 0&0&0 \end{array} \right )
$$
and
$$
J = \left ( \begin{array}{cccccccc}  1&-1& 0& 0& 0&0& b_1& b_2\\[3pt]
-1& -1& -b_2& b_1& 0& 0&0&0\\[3pt]   0& \frac{3}{b_2}&  1&  0&-b_1&  0 &  0 & -1\\[3pt]   0& 0 & 0& 1& -b_2 &  0 &  1 &  0\\[3pt]   0& 0& 0 & \frac{3}{b_2} &  -1 &  1 & 0&0\\[3pt]  0& 0& 0& 0& 1& 1& - \frac{3}{b_2} & 0\\ 0& 0& 0& 1& 0& b_2&-1 &0\\[3pt]  - \frac{3}{b_2}& 0&-1& 0& 0&-b_1 &0& -1 \end{array} \right),
$$
where $b_{1}$ and $b_{2} \neq 0$ are real numbers.   It easy to check that $I$ and $J$ are both Hermitian with respect to the  standard neutral metric $h$. Moreover, the  generalized hypercomplex ${\mathcal I}_{a,b,c}$  is defined by the matrix
{\small $$
 \left ( \begin{array}{cccccccc}  b+ 2c &  -a-b + 2 c& 2 c b_2&  -2c b_1&  0&  0& b b_1&  b b_2\\[3pt]
 a-b + 2 c&  -b - 2 c& -b b_2&  b b_1&  0&  0&  2 c b_1& 2 c b_2\\[3pt]
 \frac{- 6 c}{b_2}&  \frac{3 b}{b_2}&b - 2 c& 0& - b b_1&- 2 c b_1&0& a - b - 2 c\\[3pt]
 0& 0& 0& b - 2 c&-b b_2&-2 c b_2&-a + b + 2 c& 0\\[3pt]
0& 0&  0 & \frac{3 b}{b_2}& - b - 2 c&-a + b - 2 c& \frac{6 c}{b_2}& 0\\[3pt]
0& 0& 0 &  \frac{6 c}{b_2}& a + b - 2 c& b + 2 c& \frac{- 3 b}{b_2}& 0 \\[3pt]
 0&0& 0&a + b +  2 c&- 2 c b_2& b b_2&-b + 2 c&0\\[3pt]
\frac{-3 b}{b_2}& \frac{-6 c}{b_2}& -a - b- 2 c&0&2 c b_1& - b b_1&0&- b + 2 c \end{array} \right )
$$}
and is of type $1$.

So we obtain the following:
\begin{prop}
The Kodaira-Thurston surface admits both a holomorphic symplectic structure and a generalized hypercomplex structure in which the general complex structures in the generalized hypercomplex family are of type $1$.

\end{prop}

\smallskip
\textbf{Acknowledgements}. 
 Anna Fino is  partially supported by Project PRIN 2022 \lq \lq Geometry and Holomorphic Dynamics”, by GNSAGA (Indam) and  by a grant from the Simons Foundation (\#944448). 
Gueo Grantcharov is partially supported by a grant from the Simons Foundation (\#853269).


\smallskip


\begin{thebibliography}{12}

\bibitem{AD} D.V. Alekseevsky, L. David, A note about invariant SKT structures and generalized K\"ahler structures on flag manifolds, Proc. Edinb. Math. Soc. (2012),  543--549.

\bibitem{AGJ} D. Alvarez, M. Gualtieri, Y. Jiang, Symplectic double groupoids and the generalized Kähler potential, preprint, arXiv:2407.00831
 


\bibitem{ABDF}  L.C. de Andr\'es, M. L. Barberis, I. Dotti, M. Fern\'andez,  Hermitian structures on cotangent bundles of four dimensional solvable Lie groups, Osaka J. Math.  44 (2007),  765--793.


\bibitem{A-C-K}D. Angella, S. Calamai, H. Kasuya, Cohomologies of generalized complex manifolds and nilmanifolds,
J. Geom. Anal. 27 (2017), no. 1, 142–161.


\bibitem{AGG} V. Apostolov, P. Gauduchon, G. Grantcharov, Bihermitian structures on complex surfaces, Proc. London Math. Soc. (3) 79 (1999), 414--428. Corrigendum 92 (2006), no. 1, 200--202

\bibitem{AG} V. Apostolov, M. Gualtieri, Generalized K\"ahler manifolds, commuting complex structures, and split tangent bundles, Comm. Math. Phys. 271 (2007), no. 2, 561--575.


\bibitem{AS}  V.  Apostolov,  J. Streets,  The nondegenerate generalized K\"ahler Calabi-Yau problem, J. Reine Angew. Math. 777 (2021), 1--48.


\bibitem{ASU}  V. Apostolov, J. Streets,   Y. Ustinovskiy, Variational structure and uniqueness of generalized K\"ahler-Ricci solitons, Peking Math. J. 6 (2023), no. 2, 307–351.

\bibitem{BCD}  M. A.  Bailey,  G. R. Cavalcanti,  J. L. van der Leer Dur\'an,  Blow-ups in generalized complex geometry, Trans. Amer. Math. Soc. 371 (2019), no. 3, 2109--2131.


\bibitem{B-G}M. Bailey, M. Gualtieri, Integration of generalized complex structures,
J. Math. Phys. 64 (2023), no. 7, Paper No. 073503, 24 pp.


\bibitem{B-G-2} M. Bailey, M. Gualtieri, Local analytic geometry of generalized complex structures'
Bull. Lond. Math. Soc. 49 (2017), no. 2, 307–319.

\bibitem{Brendhaber}  A. Bredthauer,  Generalized Hyperk\"ahler Geometry and Supersymmetry, arXiv:0608114.

\bibitem{Br} A. Blaga, A. Nannicini,  Canonical connections attached to generalized hypercomplex and biparacomplex structures,
Rev. R. Acad. Cienc. Exactas Fís. Nat. Ser. A Mat. RACSAM 117 (2023), no.4, Paper No. 150, 30 pp.

  \bibitem{BM} M. Boucetta, M. W. Mansouri, Left invariant generalized complex and K\"ahler structures on simply connected four dimensional Lie groups: classification and invariant cohomologies,  J. Algebra 576 (2021), 27--94.

\bibitem{BF}  B. Brienza, A. Fino, Generalized K\"ahler manifolds via mapping tori, arXiv:2305.11075, to appear in J. Symplectic Geom.

\bibitem{BFG}  F. Brienza, A. Fino,  G. Grantcharov, CYT and SKT manifolds with parallel Bismut torsion, arXiv:2401.07800,   to appear in Proc. Roy. Soc. Edinburgh Sect.A.

\bibitem{Cavalcanti} G. R. Cavalcanti, Formality in generalized K\"ahler geometry, Topol. Appl. 154 (2007), 1119--1125.


\bibitem{CG} G. R. Cavalcanti, M. Gualtieri, Blowing up generalized K\"ahler 4-manifolds, Bull. Braz. Math. Soc.
(N.S.) 42 (2011), 537--557.





\bibitem{C-M} G.  R. Cavalcanti, M. Gualtieri, Stable generalized complex structures,
Proc. Lond. Math. Soc. (3) 116 (2018), no. 5, 1075–1111.



\bibitem{C-K-W} G.  R. Cavalcanti, R. Klaasse, A. Witte, Fibrations in semitoric and generalized complex geometry,
Canad. J. Math. 75 (2023), no. 2, 645–685.

\bibitem{C-K-W-2} G.  R. Cavalcanti, R. Klaasse, A. Witte, Self-crossing stable generalized complex structures,
J. Symplectic Geom. 20 (2022), no. 4, 761–811.


\bibitem{C-N} H. Chen, X. Nie, Odd type generalized complex structures on 4-manifolds.
J. Geom. Anal. 31 (2021), no. 1, 457–474.



\bibitem{C-D}  V. Cortes, L. David,
Generalized connections, spinors, and integrability of generalized structures on Courant algebroids,
Mosc. Math. J. 21 (2021), no. 4, 695--736.

 \bibitem{DGMY}  J. Davidov,  G. Grantcharov,  O. Mushkarov,  M.  Yotov,  Generalized pseudo-K\"ahler structures,
Comm. Math. Phys. 304 (2011), no. 1, 49–68.

\bibitem{DM} J. Davidov, O. Mushkarov, Twistorial construction of generalized K\"ahler manifolds, J. Geom. Phys.
57 (2007), 889--901.

\bibitem{Deschamps}  G. Deschamps, Twistor space of a generalized quaternionic manifold,
Proc. Indian Acad. Sci. Math. Sci.131 (2021), no.1, Paper No. 1, 20 pp.

\bibitem{FP} A. Fino, F. Paradiso, Generalized K\"ahler almost abelian Lie groups, Ann. Mat. Pura Appl. (4) 200 (2021), no. 4, 1781--1812.

\bibitem{FP2} A. Fino, F. Paradiso, Hermitian structures on a class of almost nilpotent solvmanifolds, J. Algebra
609 (2022), 861--925.

\bibitem{FT} A. Fino, A. Tomassini, Non-K\"ahler solvmanifolds with generalized K\"ahler structure, J. Symplectic Geom. 7 (2009), 1--14.

\bibitem{Mario-Jeff} M. Garcia-Fernandez, J. Streets,  Generalized Ricci flow, Univ. Lecture Ser., 76
American Mathematical Society, Providence, RI, 2021, vi+248 pp.


\bibitem{G-V-GV} E. Gasparim, F. Valencia, C. Vares, Invariant generalized complex geometry on maximal flag manifolds and their moduli,
J. Geom. Phys. 163 (2021), Paper No. 104108, 21 pp.



\bibitem{GHR} S. Gates, C. Hull, M. Ro\v cek,   Twisted multiplets and new supersymmetric nonlinear $\sigma$-models, Nuclear Phys. B 248 (1984), no. 1, 157--186.

\bibitem{GS} R. Glover, J. Sawon, Generalized twistor spaces for hyperk\"ahler manifolds,
J. Lond. Math. Soc. (2) 91 (2015), no. 2, 321--342.



\bibitem{Goto} R. Goto, Unobstructed deformations of generalized complex structures induced by $\mathcal C^{\infty}$ logarithmic symplectic structures and logarithmic Poisson structures,
Springer Proc. Math. Stat., 154
Springer, [Tokyo], 2016, 159–183.


\bibitem{Goto2}  R. Goto,  Scalar curvature as moment map in generalized K\"ahler geometry, J. Symplectic Geom. 18 (2020), no. 1, 147--190.




\bibitem{G-S}L. Grama, L. Soriani, A remark about mirror symmetry of elliptic curves and generalized complex geometry,
Proyecciones 42 (2023), no. 2, 445–456.


\bibitem{Gualtieri}  M. Gualtieri,  Generalized complex geometry, Ann. of Math. (2) 174 (2011), no. 1, 75--123.


\bibitem{Hitchin} N. Hitchin, Generalized Calabi-Yau manifolds, Quart. J. Math. Oxford Ser. 54 (2003), no. 3, 281--308.


\bibitem{St2} W. Hong, M. Stienon, From Hypercomplex to Holomorphic Symplectic Structures,  J. Geom. Phys. 96 (2015), 187--03.

\bibitem{Hull-Zabzine} C. Hull, M. Zabzine, N=(2,2) superfields and geometry revisited, preprint, arXiv:2404.19079.


\bibitem{Kaledin}   D. Kaledin, Integrability of the twistor space for a hypercomplex manifold, Selecta Math. (N.S.)4(1998), no.2, 271--278.

\bibitem{KSS} T. Kimura,  S. Sasaki,  K.  Shiozawa,
Hyperk\"ahler, bi-hypercomplex, generalized hyperk\"ahler structures and T-duality, Nuclear Phys. B 981 (2022), Paper No. 115873, 19 pp.




\bibitem{Kob} K. Kobayashi, On a B-field transform of generalized complex structures over complex tori
J. Geom. Phys. 206 (2024).



\bibitem{Mun} U.R. Mun,  Some examples of stable generalized complex 6-manifolds with either 0 or negative Euler characteristic,
Differential Geom. Appl. 78 (2021), Paper No. 101799, 11 pp.


\bibitem{Pantilie} R. Pantilie, Generalized quaternionic manifolds, Ann. Mat. Pura Appl. 193  (2014), 633--641.

\bibitem{Pap-W} G. Papadopoulos, E. Witten, Scale and Conformal Invariance in 2d Sigma Models, with an Application to N=4 Supersymmetry, preprint, arXiv:2404.19526.

\bibitem{Pap} G.  Papadopoulos, Scale and Conformal Invariance in Heterotic $\sigma$-Models, preprint, arXiv:2409.01818.


\bibitem{Sil} L. Sillari, Generalized Luttinger surgery and other cut-and-paste constructions in generalized complex geometry,
J. Geom. Phys. 194 (2023), Paper No. 105017, 14 pp.

\bibitem{St1} M. Stienon,  Hypercomplex structures on Courant algebroids, C. R. Math. Acad. Sci. Paris 347 (2009), no.9-10, 545--550.



\bibitem{Va} C. Varea, Invariant generalized complex structures on partial flag manifolds,
Indag. Math. (N.S.) 31 (2020), no. 4, 536–555.



\bibitem{V-SM} C. Varea, L. San Martin, Invariant generalized complex structures on flag manifolds,
J. Geom. Phys. 150 (2020), 103610, 17 pp.


\bibitem{Wang}  Y. Wang, Toric generalized K\"ahler structures. II, J. Symplectic Geom. 21 (2023), no. 2, 235--264.

\bibitem{Witten}  E. Witten, Instantons and the Large $N = 4$ Algebra,  arXiv:2407.20964. 









\end{thebibliography}
\end{document}